\documentclass{amsart}

\usepackage{amsmath,amsthm,amssymb,amsfonts}
\usepackage{bbm}
\usepackage[mathscr]{eucal}
\usepackage[english]{babel}
\usepackage{graphicx}

\allowdisplaybreaks[1]

\newtheorem{thm}{Theorem}[section]
\newtheorem{cor}[thm]{Corollary}
\newtheorem{lem}[thm]{Lemma}
\newtheorem{prop}[thm]{Proposition}
\newtheorem{rem}[thm]{Remark}

\theoremstyle{remark}
\theoremstyle{definition}

\numberwithin{equation}{section}

\newcommand{\thmref}[1]{Theorem~\ref{#1}}
\newcommand{\lemref}[1]{Lemma \ref{#1}}

\newcommand{\remref}[1]{Remark \ref{#1}}

\def\CC{{\mathbb C}}
\def\NN{{\mathbb N}}
\def\RR{{\mathbb R}}
\def\SS{{{\mathbb S}^{d-1}}}

\def\ZZ{{\mathbb Z}}

\def\cF{\mathcal{F}}

\def\fD{{\mathfrak D}}

\def\sF{\mathscr{F}}
\def\sG{\mathscr{G}}

\def\HHH{\mathscr{H}}

\def\eps{{\varepsilon}}
\def\u*{*}

\def\SSS{{\mathbb S}}
\def\FF{\sF}
\def\GG{G}
\def\GGG{g}

\begin{document}

\title[Localized kernels in terms of Newtonian kernels]
{Highly Localized Kernels on the Sphere \\
Induced by Newtonian kernels}

\author{Kamen Ivanov}
\address{Institute of Mathematics and Informatics\\
Bulgarian Academy of Sciences\\ 1113 Sofia\\ Bulgaria}
\email{kamen@math.bas.bg}

\author{Pencho Petrushev}
\address{Department of Mathematics\\University of South Carolina\\
Columbia, SC 29208}
\email{pencho@math.sc.edu}

\thanks{The first author has been supported by Grant DN 02/14
of the Fund for Scientific Research of the Bulgarian Ministry of Education and Science.
The second author has been supported by NSF Grant DMS-1714369.}

\begin{abstract}
The purpose of this article is to construct highly localized summability kernels on the unit sphere in $\RR^d$
that are restrictions to the sphere of linear combinations
of a small number of shifts of the fundamental solution of the Laplace equation
(Newtonian kernel) with poles outside the unit ball in $\RR^d$.
The same problem is also solved for the subspace $\RR^{d-1}$ in $\RR^d$.

\end{abstract}

\date{August 26, 2018}

\maketitle

\section{Introduction}\label{s1}

The shifts of the fundamental solution of the Laplace equation
$\frac{1}{|x|^{d-2}}$ in dimensions $d>2$ or
$\ln \frac{1}{|x|}$ if $d=2$
with $|x|$ being the Euclidean norm of $x\in \RR^d$
are basic building blocks in Potential theory.
As is customary, we shall term the harmonic function
$\frac{1}{|x|^{d-2}}$ or $\ln \frac{1}{|x|}$
``Newtonian kernel''.

We are interested in the problem for approximation
of harmonic functions on the unit ball $B^d$ in $\RR^d$ from finite linear combinations
of shifts of the Newtonian kernel.
More explicitly, the problem is
for a given harmonic function $U$ on $B^d$ and $n\ge 1$ to find
$n$ locations $\{y_j\}$ in $\RR^d\setminus \overline{B^d}$
and coefficients $\{c_j\}$ in $\CC$ such that
\begin{equation}\label{lin-comb}
c_0+\sum_{j=1}^n \frac{c_j}{|x-y_j|^{d-2}}\quad\hbox{if} \;\; d>2
\quad\hbox{or}\quad
c_0+\sum_{j=1}^n c_j\ln \frac{1}{|x-y_j|}\quad\hbox{if} \;\; d=2
\end{equation}
approximates $U$ well (near best) in the harmonic Hardy space $\HHH^p(B^d)$, $0<p\le \infty$.

This problem is also important in the case when $U$ is harmonic on
$\RR^d\setminus \overline{B^d}$ and the poles $\{y_j\}$ are in $B^d$ or
$U$ is harmonic on
$\RR^d_+$ and the poles $\{y_j\}$ are in $\RR^d_-$.

An alternative formulation of the problem is to approximate a given potential $U$
by the potential of $n$ point masses (using terminology from Geodesy)
or by the potential of $n$ point charges (in terms of Electrostatics)
or by the potential of $n$ magnetic poles (in Magnetism).

It should be pointed out that there is a great deal of work done on the Method of Fundamental Solutions
for the Dirichlet problem of the Laplace equation in Numerical Analysis.
This theme is directly related to the problems we consider here.
We refer the reader to \cite{AG, Helms, K} for the basics of Potential theory.

The poor localization of the Newtonian kernel makes the above approximation problem
unamenable and challenging.
An important step in solving this problem (see \cite{IP1}) is to construct highly localized summability kernels on the unit sphere
$\SS$ in $\RR^d$ that are restrictions to the sphere of linear combinations of finitely many (fixed number)
shifts of the Newtonian kernel.
This is the main goal of this article.

The simple fact that
\begin{equation}\label{norm-rep}
|x-a\eta|^2=a^2+1-2a(x\cdot\eta),\quad x,\eta\in\SS,
\end{equation}
implies that the restriction of any shift of the Newtonian kernel to $\SS$ is a zonal function,
i.e. it is the composition $F(x\cdot\eta)$ of an appropriate univariate function $F:[-1,1]\to\RR$
and the dot product $x\cdot\eta$, $x,\eta\in\SS$.
This leads us to the following explicit formulation of the problem at hand:

\smallskip


\noindent
{\bf Problem 1.}
{\em
Let $M>d-1$. For given $\eps\in(0,1]$  find $2m+1$ constants
$b_\nu\in\RR$, $a_\nu>1$ so that the restriction $F_\eps(x\cdot\eta)$ of the function
\begin{equation}\label{cond-3}
f_{\eps,\eta}(x) = b_0 + \sum_{\nu=1}^{m} \frac{b_\nu}{|x-a_\nu\eta|^{d-2}},
\quad \eta\in\SS,\; x\in \RR^d\setminus\cup_{\nu=1}^m\{a_\nu\eta\},
\quad\hbox{if $d>2$},
\end{equation}
or
\begin{equation}\label{cond-4}
f_{\eps,\eta}(x) = b_0 + \sum_{\nu=1}^{m} b_\nu\ln \frac{1}{|x-a_\nu\eta|},
\quad \eta\in{\mathbb S}^1,\; x\in \RR^2\setminus\cup_{\nu=1}^m\{a_\nu\eta\},
\quad\hbox{if $d=2$,}
\end{equation}
to the unit sphere $\SS \subset \RR^d$ satisfies the following conditions:
\begin{equation}\label{cond-1}
|F_\eps(x\cdot\eta)| \le \frac{c\eps^{-d+1}}{(1+\eps^{-1}\rho(x, \eta))^M},
\quad\forall x,\eta\in \SS;
\end{equation}
\begin{equation}\label{cond-2}
\int_{\SS}F_\eps(x\cdot\eta)d\sigma(x)=1, \quad\forall \eta\in\SS
\end{equation}
with constants $m\in\NN$ and $c>0$ depend only on $M$ and $d$. }

Here
$\rho(x, \eta):=\arccos \,(x\cdot \eta)$ is the geodesic distance between $x,\eta\in\SS$ and
$\sigma$ denotes the Lebesgue measure on $\SS$.
It should be pointed out that the localization required in \eqref{cond-1}--\eqref{cond-2} is only
\emph{on the boundary} $\SS$ of the unit ball.
As far as every such $f_{\eps,\eta}$ is a \emph{harmonic} function on $B^d$
it cannot be well localized in the interior of the ball.

We shall present two solutions (even three in dimension $d=2$) of Problem~1.
To solve this problem it suffices to solve either of the following two problems:

\smallskip


\noindent
{\bf Problem 2.}
{\em
Let $M>d-1$. For given $\eps\in(0,1]$ find constants $a_j>1$ and $b_j, c_j\in\RR$
so that the restriction $F_\eps(x\cdot\eta)$ of the function
\begin{equation}\label{cond-231}
f_{\eps,\eta}(x) = \sum_{j=1}^{m} \frac{b_j}{|x-a_j\eta|^{d-2}}
+ \sum_{j=1}^{m} c_j(\eta\cdot \nabla)\Big(\frac{1}{|x-a_j\eta|^{d-2}}\Big),
\end{equation}
$\eta\in\SS$, $x\in \RR^d\setminus \{a_1\eta, \dots, a_m\eta\}$,
if $d>2$
or
\begin{equation}\label{cond-241}
f_{\eps,\eta}(x) = b_0 + \sum_{j=1}^{m} c_j(\eta\cdot \nabla)\ln\frac{1}{|x-a_j\eta|},
\;\; \eta\in{\mathbb S}^1,\; x\in \RR^d\setminus \{a_1\eta, \dots, a_m\eta\},
\;\;\hbox{if $d=2$,}
\end{equation}
to $\SS$ satisfies conditions \eqref{cond-1}--\eqref{cond-2},
where as above the constants $m\in\NN$ and $c>0$ depend only on $M$ and $d$.
}

\smallskip


\noindent
{\bf Problem 3.}
{\em
Let $M>d-1$. For given $\eps\in(0,1]$ find $m+1$ constants $b_\ell\in\RR$ and $a>1$
so that the restriction $F_\eps(x\cdot\eta)$ of the function
\begin{equation}\label{cond-23}
f_{\eps,\eta}(x) = \sum_{\ell=0}^{m} b_\ell(\eta\cdot \nabla)^\ell\Big(\frac{1}{|x-a\eta|^{d-2}}\Big),
\quad \eta\in\SS, \; x\in \RR^d\setminus \{a\eta\},
\quad\hbox{if $d>2$};
\end{equation}
or
\begin{equation}\label{cond-24}
f_{\eps,\eta}(x) = b_0 + \sum_{\ell=1}^{m} b_\ell(\eta\cdot \nabla)^\ell\ln\frac{1}{|x-a\eta|},
\quad \eta\in{\mathbb S}^1,\; x\in \RR^2\setminus \{a\eta\},
\quad\hbox{if $d=2$,}
\end{equation}
to $\SS$ satisfies conditions \eqref{cond-1}--\eqref{cond-2},
where as above the constants $m\in\NN$ and $c>0$ depend only on $M$ and $d$.
}

\smallskip

As is well known the $\ell$th directional derivative operator $(\eta\cdot\nabla)^\ell$,
where $\nabla$ stands for the gradient operator,
is approximated well by the finite difference operator
$\fD^\ell_t(\eta):=t^{-\ell}\sum_{k=0}^\ell (-1)^{\ell-k} \binom{\ell}{k} T(\eta,kt)$,
where $T(\eta,t)f(x):=f(x+t\eta)$, $x\in\RR^d$.
More precisely, if $d>2$, $\ell\ge 1$, $a>1$, and $\eta\in\SS$, then
$$
\big\|(\eta\cdot\nabla)^\ell|x-a\eta|^{2-d}
- \fD^\ell_t(\eta)|x-a\eta|^{2-d}\big\|_{L^\infty(\overline{B^d})} \to 0
\quad\hbox{as}\quad t \to 0,
$$
and a similar claim is valid when $d=2$.
Having in mind that $\fD^\ell_t(\eta)|x-a\eta|^{2-d}$ is
a linear combination of Newtonian kernels with poles at $(a-kt)\eta$, $k=0,\dots,\ell$,
we see that, a solution of Problem~2 or Problem~3 leads immediately to a solution of Problem~1.

It is easy to see that a properly dilated and normalized version of the Poisson kernel
provides a solution of Problem~2 and Problem~3 in the case $M=d$.
Indeed, the Poisson kernel for a ball of radius $a>1$ in $\RR^d$ takes the form
\begin{equation}\label{Poisson-rep}
P(y,x)=\frac{1}{a\omega_d}\frac{a^2-|x|^2}{|x-y|^d}, \quad |y|=a, \quad |x|<a,
\end{equation}
where
$\omega_d:=2\pi^{d/2}/\Gamma(d/2)$ is the Lebesgue measure of $\SS$.
Restricting $P(y,x)$ to $\SS$ as a function of $x$ and setting $y=a\eta$ with $\eta\in\SS$ and $a:=1+\eps$
we get
$P(a\eta,x)=\frac{1}{a\omega_d}\frac{a^2-1}{|x-a\eta|^d}$.
A straightforward derivation shows that
\begin{equation}\label{Poisson1}
(\eta\cdot \nabla)|x-a\eta|^{2-d} = (d-2)(2a)^{-1}|x-a\eta|^{2-d}+2^{-1}(d-2)\omega_d P(a\eta,x),
\;\;\hbox{if}\; d>2.
\end{equation}
Hence, the kernel $F_\eps(x\cdot\eta):=P(a\eta,x)$ is of the forms (\ref{cond-231}) and (\ref{cond-23}) with $m=1$.
It is also easy to see that in dimension $d=2$
\begin{equation}\label{Poisson2}
(\eta\cdot \nabla)\ln\frac{1}{|x-a\eta|} = \frac{1}{2a}+\pi P(a\eta,x).
\end{equation}
and hence the kernel $F_\eps(x\cdot\eta):=P(a\eta,x)$ is of the forms (\ref{cond-241}) and (\ref{cond-24}) with $m=1$.

Furthermore, it is easy to show that (see (\ref{simple-ineq}))
\begin{equation}\label{dist-x-aeta}
5^{-1}(\eps+\rho(x, \eta)) \le |x-a\eta| \le 2(\eps+\rho(x, \eta)),\;\;x\in\SS,
\quad\hbox{if $\;0<\eps\le 1$.}
\end{equation}
Therefore,
$0< F_\eps(x\cdot\eta)\le c\eps^{-d+1}(1+\eps^{-1}\rho(x, \eta))^{-d}$
and hence $F_\eps(x\cdot\eta):=P(a\eta,x)$ solves Problem~2 and Problem~3 for $M=d$.

To solve Problem 2 or Problem 3 for an arbitrary $M>d$ is not so easy.
When trying to solve Problem~3 in the general case the first question that occurs is whether
the $m$th directional derivative
$(\eta\cdot\nabla)^m|x-a\eta|^{2-d}$ if $d>2$
or
$(\eta\cdot\nabla)^m\ln 1/|x-a\eta|$ if $d=2$
for sufficiently large $m$, depending on $M$, can solve the problem.
The well known Maxwell formula (see e.g. \cite[p. 479, ex. 13]{AAR}) asserts that
if $d\ge 1$, $\eta\in\SS$, $\lambda>0$, $m\in\NN$, then
\begin{equation}\label{Maxwell-1}
\left(\eta\cdot\nabla\right)^m \frac{1}{|x|^\lambda}= (-1)^m m!
C_m^{(\lambda/2)}\left(\frac{x\cdot\eta}{|x|}\right) \frac{1}{|x|^{\lambda+m}},
\quad x\in\RR^d\backslash\{0\},
\end{equation}
where $C_m^{(\mu)}$ is the $m$th degree ultraspherical polynomial normalized by the identity
$C_m^{(\mu)}(1)=\binom{m+2\mu-1}{m}$.
Now, using that
$\lim_{\mu\to 0+}\mu^{-1}(|t|^{-\mu}-1)=\ln \frac{1}{|t|}$~and
$\lim_{\mu\to 0+}\mu^{-1}C_m^{(\mu)}(t)=2m^{-1}T_m(t)$
one obtains by letting $\mu\to 0$ in (\ref{Maxwell-1})
\begin{equation}\label{Maxwell-2}
\left(\eta\cdot\nabla\right)^m \ln \frac{1}{|x|}=(-1)^{m}(m-1)! T_m\left(\frac{x\cdot\eta}{|x|}\right) \frac{1}{|x|^m},
\quad x\in\RR^2\backslash\{0\},
\end{equation}
where $T_m$ is the $m$th degree Chebyshev polynomial of the first kind normalized by $T_m(1)=1$.
Let $\eta\in\SS$, $a:=1+\eps$, $\eps>0$, and $m\in\NN$. Then (\ref{Maxwell-1})-(\ref{Maxwell-2}) yield
\begin{align}
\left(\eta\cdot\nabla\right)^m \frac{1}{|x-a\eta|^{d-2}}&= (-1)^m m!
	C_{m}^{(d/2-1)}\left(\frac{(x-a\eta)\cdot \eta}{|x-a\eta|}\right)\frac{1}{|x-a\eta|^{d-2+m}},\;\; d>2,\label{Maxwell-3}\\
\left(\eta\cdot\nabla\right)^m \ln \frac{1}{|x-a\eta|}&= (-1)^{m}(m-1)!
	T_{m}\left(\frac{(x-a\eta)\cdot \eta}{|x-a\eta|}\right)\frac{1}{|x-a\eta|^m}, \;\;d=2.\label{Maxwell-4}
\end{align}
Now, using (\ref{Maxwell-3}) and (\ref{dist-x-aeta}) we obtain the sharp estimate
\begin{equation}\label{Local-0}
\left|\left(\eta\cdot\nabla\right)^m \frac{1}{|x-a\eta|^{d-2}}\right|
\le c(m, d)\frac{\eps^{-m+1}\eps^{-d+1}}{(1+\eps^{-1}\rho(x, \eta))^{m+d-2}},
\quad x\in\SS.
\end{equation}
On the other hand,
since $\left(\eta\cdot\nabla\right)^m |x-a\eta|^{2-d}$ is a harmonic function we have
\begin{align*}
\int_\SS \left(\eta\cdot\nabla\right)^m \frac{1}{|x-a\eta|^{d-2}} d\sigma(x)
&= \omega_d\left(\eta\cdot\nabla\right)^m \frac{1}{|x-a\eta|^{d-2}}\Big|_{x=0}
\\
= \omega_d m!C_{m}^{(d/2-1)}(1)a^{-d+2-m}
&= \omega_d m!\binom{m+d-3}{m}a^{-d+2-m}
=\frac{c(m, d)}{(1+\eps)^{m+d-2}}.
\end{align*}
Therefore, if we set
$$
F_\eps(x\cdot\eta):=c^*\left(\eta\cdot\nabla\right)^m |x-a\eta|^{2-d}
$$
with a normalization constant $c^*$ so that $F_\eps(x\cdot\eta)$ obeys (\ref{cond-2})
then in light of the additional multiplier $\eps^{-m+1}$ in \eqref{Local-0}
$|F_\eps(x\cdot\eta)|$ with $m\ge 2$ cannot have the decay from \eqref{cond-1} for any $M>d-1$.
The same argument applies if $d=2$.
The conclusion is that Problem~2 cannot be solved by using a single $m$th directional derivative
of the Newtonian kernel.

In this article we present two main results.
First, modifying Lemma 2.5 in L.~Colzani \cite{Colzani} we show that the function
$$
F_\eps(x\cdot\eta) = \sum_{j=1}^m (-1)^{j+1} \binom{m}{j} (1+j\eps)^{d-1}P((1+j\eps)\eta, x),
$$
where $P$ is the Poisson kernel \eqref{Poisson-rep} and $m \ge M-d$,
solves Problem~2.
Secondly, we show that Problem~3 is solved by the simpler kernel
$$
F_\eps(x\cdot\eta) := \frac{c^\star\eps^{2m-1}}{|x-a\eta|^{2m+d-2}}, \;\;x\in\SS,
\quad\hbox{with}\quad
\int_\SS F_\eps(x\cdot\eta)d\sigma(x)=1,
$$
where
$m\ge (M-d+2)/2$,
$a=1+\eps$, and $c^\star>0$ is a normalization constant.
While the proof of the first result is straightforward, the proof of the second (more surprising) result is quite involved
and this is the main novelty in this paper.

Our solution of Problem~3 (and hence of Problem~1) has an obvious advantage over Colzani's solution of Problem~2 -
it is amenable to generalizations.
Our scheme can be used for the solution of the analog of Problem~3
and consequently Problem~1 for domains with much more complicated geometry
than the ball, while Colzani's solution of Problem~2 relying on the Poisson kernel
is limited to domains for which the Poisson kernel is
available in a convenient concrete form.

The rest of this article is organized as follows.
In Section~\ref{sec:colzani} we presents a solution of Problem~2 based on an idea of L. Colzani from \cite{Colzani}.
In Section~\ref{s4_2} we present the solution of Problem~3 mentioned above.
In Section~\ref{sec:S2} we present a second solution of Problem~3 in dimension $d=2$.
Section~\ref{sec:S4} treats in brief the localization on $\SS$ of harmonic functions on $\RR^d\setminus \overline{B^d}$.
As a~natural progression of our development, in Section \ref{sec:Rd} we also solve the analogues of Problems 2 and 3 and as consequence
the analogue of Problem~1
with $\SS$ replaced by $\RR^{d-1}$.

\section{Localized kernels on $\SS$ in terms of Newtonian kernels: Solution of Problem 2}\label{sec:colzani}

In this section we present a solution of Problem~2 from \S\ref{s1} based on the idea from \cite[Lemma 2.1]{Colzani}.


\begin{thm}\label{thm:colzani}
Let $m\in\NN$, $d\ge 2$, $\eta\in\SS$, and $0<\eps\le 1$.
Consider the function
\begin{equation}\label{def-f-eps}
f_{\eps, \eta}(x) := \sum_{j=1}^m (-1)^{j+1} \binom{m}{j} (1+j\eps)^{d-1}P((1+j\eps)\eta, x),
\quad x\in \RR^d\setminus\cup_{j=1}^m\{(1+j\eps)\eta\},
\end{equation}
where $P$ is the Poisson kernel \eqref{Poisson-rep}.
Then the restriction $F_{\eps}(x\cdot \eta)$ of the function $f_{\eps, \eta}(x)$ on $\SS$ has these properties:
\begin{equation}\label{colzani-1}
|F_{\eps}(x\cdot\eta)| \le \frac{c\eps^{-d+1}}{(1+\eps^{-1}\rho(x, \eta))^{m+d-1}},
\quad\forall x, \eta\in \SS,
\end{equation}
and
\begin{equation}\label{colzani-2}
\int_{\SS} F_{\eps}(x\cdot\eta)d\sigma(x) =1,
\quad\forall \eta\in \SS,
\end{equation}
where $c>0$ is a constant depending only on $m$ and $d$.
Furthermore, $f_{\eps, \eta}(x)$ can be represented in the form
\begin{equation}\label{col-3}
f_{\eps,\eta}(x) = \sum_{j=1}^{m} \frac{b_j}{|x-a_j\eta|^{d-2}}
+ \sum_{j=1}^{m} c_j(\eta\cdot \nabla)\Big(\frac{1}{|x-a_j\eta|^{d-2}}\Big),
\quad\hbox{if}\quad d>2,
\end{equation}
or
\begin{equation}\label{col-4}
f_{\eps,\eta}(x) = b_0 + \sum_{j=1}^{m} c_j(\eta\cdot \nabla)\ln\frac{1}{|x-a_j\eta|},
\quad\hbox{if} \quad d=2,
\end{equation}
where $a_j:=1+j\eps$.

\end{thm}

\begin{proof}
From the definition of $f_{\eps,\eta}(x)$ and \eqref{Poisson1}--\eqref{Poisson2} it readily follows
$f_{\eps,\eta}(x)$ can be represented in the form \eqref{col-3} or \eqref{col-4}.

From the harmonicity of the Poisson kernel we know that
$\int_\SS P(a\eta, x)d\sigma(x)=\omega_d P(a\eta, 0)=a^{-d+1}$, $a>1$, implying
$$
\int_\SS f_{\eps, \eta}(x)d\sigma(x)= \sum_{j=1}^m (-1)^{j+1} \binom{m}{j} =1,
$$
which confirms \eqref{colzani-2}.

To prove \eqref{colzani-1} we first observe that for $x, \eta\in \SS$ and $a>1$ (see \eqref{norm-rep})
\begin{equation}\label{rep-norm}
|x-a\eta|^2
= (1-a)^2+a\sin^2(\beta/2)
\quad\hbox{with}\;\; \beta:=\rho(x,\eta),
\end{equation}
and hence, using \eqref{Poisson-rep},
\begin{equation}\label{Rep-Poisson}
P(a\eta, x)=\frac{1}{a\omega_d}\frac{a^2-1}{[(a-1)^2+a\sin^2(\beta/2)]^{d/2}}.
\end{equation}
If $\rho(x, \eta)\le \eps$, then from above it readily follows that
$|P((1+j\eps)\eta, x)|\le c\eps^{-d+1}$.
This and the definition of $f_{\eps, \eta}(x)$ yield \eqref{colzani-1}.

Let $\rho(x, \eta)> \eps$. Clearly, $P(\eta, x)=0$ since $x, \eta\in \SS$, $x\ne \eta$.
Hence,
$$
f_{\eps, \eta}(x) = \sum_{j=0}^m (-1)^{j+1} \binom{m}{j} (1+j\eps)^{d-1}P((1+j\eps)\eta, x).
$$
Denote $g(u):= (1+u)^{d-1}P((1+u)\eta, x)$ with $x, \eta\in \SS$ fixed.
Then
$$
f_{\eps, \eta}(x)= (-1)^{m+1}\Delta_\eps^m g(0)
= (-1)^{m+1}\int_0^\eps\cdots\int_0^\eps g^{(m)}(u_1+\dots+u_m) du_1\dots du_m.
$$
We claim that
\begin{equation}\label{est-Pois}
|g^{(m)}(u)| \le \frac{c}{|x-(1+u)\eta|^{m+d-1}},
\quad 0<u<m,
\end{equation}
where $c$ is a constant depending only on $m$ and $d$.
Indeed, from \eqref{Rep-Poisson}
\begin{align*}
g(u) &= \frac{(2+u)(1+u)^{d-2}}{\omega_d}\frac{u}{(u^2+(1+u)\sin^2(\beta/2))^{d/2}}
\\
&=:\phi(u)u(u^2+(1+u)\sin^2(\beta/2))^{-d/2}.
\end{align*}
Using this representation of $g(u)$ it easily follows that \eqref{est-Pois} holds.

Finally, \eqref{est-Pois} coupled with \eqref{rep-norm} yields \eqref{colzani-1}.
\end{proof}

\section{Localized kernels on $\SS$ in terms of Newtonian kernels: Solution of Problem 3}\label{s4_2}

The solution of Problem~3 from the introduction is essentially contained in the following


\begin{thm}\label{thm:main}
Let $m\in\NN$, $d\ge 2$, $\eta\in\SS$, and $0<\eps\le 1$.
Set $a:=1+\eps$ and $\delta:=1-a^{-2}$.
Consider the function
\begin{equation*}
\cF_{\eps}(t) := \frac{(d/2)_{m-1}}{2m!}a^{2m}\delta^{2m-1} (a^2+1-2at)^{-d/2+1-m},\quad t\in[-1,1].
\end{equation*}
The function $\cF_{\eps}$ has these properties:
\begin{equation}\label{eq:2.3}
\cF_{\eps}(x\cdot\eta) = \frac{(d/2)_{m-1}}{2m!}a^{2m}\delta^{2m-1}	|x-a\eta|^{-d+2-2m},\quad x\in\SS,
\end{equation}
\begin{equation}\label{main-1}
0< \cF_{\eps}(x\cdot\eta) \le \frac{c_1\eps^{-d+1}}{(1+\eps^{-1}\rho(x, \eta))^{2m+d-2}},
\quad\forall x, \eta\in \SS,
\end{equation}
and
\begin{equation}\label{main-2}
\int_{\SS}\cF_{\eps}(x\cdot\eta)d\sigma(x) \ge c_2>0,
\quad\forall \eta\in \SS,
\end{equation}
where $c_1, c_2>0$ are constants depending only on $m$ and $d$.
Furthermore, $\cF_{\eps}(x\cdot\eta)$ is the restriction on $\SS$ of the harmonic function,
defined on $\RR^d\setminus\{a\eta\}$,
\begin{equation}\label{eq:2.1}
	\sF_{\eps,m}(a\eta,x):=q_0|x-a\eta|^{2-d}+\sum_{\ell=1}^m
	\frac{q_{\ell}\delta^{\ell-1} a^\ell}{\ell!(d-2)} (\eta\cdot\nabla)^\ell |x-a\eta|^{2-d}
\quad\mbox{if}~d\ge3,
\end{equation}
or
\begin{equation}\label{eq:2.2}
	\sF_{\eps,m}(a\eta,x):=q_0+\sum_{\ell=1}^m
	\frac{q_{\ell}\delta^{\ell-1} a^\ell}{\ell!} (\eta\cdot\nabla)^\ell \ln \frac{1}{|x-a\eta|} \quad\mbox{if}~d=2,
\end{equation}
where the coefficients $q_0,\dots, q_m$ are determined as the solution of the linear system of $m+1$ equations:
\begin{align}
	q_0+\sum_{\ell=1}^m (d/2)_{\ell-1}\frac{\delta^{\ell-1}}{2\ell!} q_{\ell}&=0,\nonumber \\
	\sum_{\ell=0}^{m-\nu}\left[\sum_{k=(\ell-\nu)_+}^{\ell}
	(-1)^{\ell-k} \binom{\nu}{\ell-k}(d/2+\nu-1)_k \frac{\delta^k}{k!}\right] q_{\nu+\ell} &=0,\label{eq:2.14b}\\
	\nu&=1,\dots,m-1,\nonumber\\
	q_m&=1.\nonumber
\end{align}
Here $(u)_0:=1$, $(u)_k:=u(u+1)\cdots(u+k-1)$
denotes the Pochhammer's symbol
and $(u)_+:=\max\{0,u\}$.
\end{thm}

\begin{rem}\label{rem:2}
In fact, the function $\sF_{\eps,m}(a\eta,x)$ from \eqref{eq:2.1} or \eqref{eq:2.2} for $x\in B^d$
is the harmonic extension of $\cF_{\eps}(x\cdot\eta)$ to $B^d$.
Note also that unlike \eqref{eq:2.1} identity \eqref{eq:2.2} contains
the constant term $q_0$ instead of a Newtonian kernel term like $q_0\ln \frac{1}{|x-a\eta|}$.
\end{rem}

Theorem~\ref{thm:main} immediately implies

\begin{cor}\label{cor-main}
Let $d \ge 2$, $M>d-1$.
Under the hypotheses of Theorem~\ref{thm:main} define
\begin{equation*}
f_{\eps,\eta}(x):=\sF_{\eps,m}(a\eta,x)\Big(\int_{\SS}\sF_{\eps,m}(a\eta,y) d\sigma(y)\Big)^{-1},
\quad x\in \RR^d\setminus\{a\eta\},
\end{equation*}
where $\sF_{\eps,m}$ is from \eqref{eq:2.1} or \eqref{eq:2.2} and $m=\left\lceil (M-d+2)/2\right\rceil$.
Then the function $f_{\eps,\eta}$ solves Problem~3 from the introduction.
\end{cor}

We shall carry out the proof of Theorem~\ref{thm:main} in three steps.

\subsection{Proof of (\ref{eq:2.3})--(\ref{main-2})}

Representation \eqref{eq:2.3} is immediate from the definition of $\cF_{\eps}$ and \eqref{norm-rep}.

We claim that
\begin{equation}\label{simple-ineq}
5^{-1}(\eps+\rho(x, \eta)) \le |x-a\eta| \le 2(\eps+\rho(x, \eta)), \quad x, \eta\in\SS.
\end{equation}
Indeed, let $x, \eta\in\SS$ and denote by $\beta$ ($0\le\beta\le\pi$) the angle between $x$ and $\eta$.
Using $\eta\cdot x=\cos\rho(x,\eta)=\cos \beta$ in \eqref{norm-rep} we get
\begin{align*}
|x-a\eta|^2 = \sin^2\beta + (a-\cos\beta)^2
= \sin^2\beta + (\eps+2\sin^2(\beta/2))^2.
\end{align*}
Assume $0\le \beta\le\pi/2$. Using the obvious inequalities
$(2/\pi)\beta \le \sin\beta \le\beta$ we obtain
$(2/\pi)^2\beta^2+\eps^2\le |x-a\eta|^2 \le \beta^2+(\eps+\beta^2/2)^2$,
which implies (\ref{simple-ineq}).
In the case $\pi/2< \beta\le\pi$ inequalities (\ref{simple-ineq}) are trivial.
Now estimate (\ref{main-1}) readily follows by (\ref{eq:2.3}) and (\ref{simple-ineq}).

Also, from \eqref{eq:2.3} we derive
\begin{multline*}
\int_{\SS}\cF_{\eps}(x\cdot\eta)\,d\sigma(x)=\omega_{d-1}\int_{-1}^1 \cF_{\eps}(u) (1-u^2)^{(d-3)/2}du\\
=\frac{(d/2)_{m-1}}{2m!}a^{2m}\delta^{2m-1}\omega_{d-1}\int_{-1}^1 (a^2+1-2au)^{(-2m-d+2)/2} (1-u^2)^{(d-3)/2}du.
\end{multline*}
Restricting the interval of integration to $[1-\eps^2,1]$ and using that $a^2+1-2au\le 5\eps^2$ for $u$ in this range we get
\begin{align*}
\int_{\SS}\cF_{\eps}(x\cdot\eta)\,d\sigma(x)
&\ge c\eps^{2m-1} \int_{1-\eps^2}^1 \eps^{-2m-d+2} (1-u)^{(d-3)/2}du
\\
&\ge c\eps^{-d+1}\eps^{2((d-3)/2+1)}=c'
\end{align*}
with $c'>0$ depending only on $d$ and $m$. This proves \eqref{main-2}.

\subsection{Solution of linear system (\ref{eq:2.14b})}

Clearly, system \eqref{eq:2.14b} has an upper triangular matrix
with 1's on the main diagonal.
Hence
$q_0,\dots, q_m$ are uniquely determined by \eqref{eq:2.14b}.

Also from \eqref{eq:2.14b} we get by induction on $\nu=m-1,m-2,\dots,0$ that the $q_\ell$'s
satisfy
\begin{equation}\label{eq:2.0}
q_{\ell}=q_{\ell}(d,m,\delta)=\sum_{k=0}^{m-\ell}\alpha_{\ell,k}\delta^k,\quad \ell=0,1,\dots,m,
\end{equation}
with some coefficients $\alpha_{\ell,k}=\alpha_{\ell,k}(d,m)$ depending only on $d$ and $m$, where
$q_m=\alpha_{m,0}=1$.
Moreover, $\alpha_{\ell,k}(d,m)$ is a polynomial of $d$ of degree $k$ and, hence, $\alpha_{\ell,0}$ does not depend on $d$.
Observe also that $\alpha_{0,m}=0$, i.e. $q_0$ is a polynomial of degree $m-1$.
The $q_\ell$'s for $m=1,2,3,4$ are given in \remref{rem:3}.

\begin{lem}\label{lem:1}
For $m\in\NN$ the numbers $\alpha_{\ell}(m):=\alpha_{\ell,0}(d,m)$, $\ell=1,\dots,m$, satisfy
\begin{equation}\label{eq:3.5}
\alpha_\ell(m)=\frac{\ell(2m-\ell-1)!}{m!(m-\ell)!},\quad \ell=1,2,\dots,m,
\end{equation}
and $\alpha_0(m)=-\alpha_1(m)/2$.
\end{lem}
\begin{proof}
The numbers $\alpha_\ell(m)$
satisfy the limit case of \eqref{eq:2.14b} when $\delta=0$,
i.e. $\alpha_0(m)+\alpha_1(m)/2=0$ and
\begin{equation}\label{eq:2.17}
	\sum_{\ell=0}^{\min\{\nu,m-\nu\}} (-1)^\ell\binom{\nu}{\ell} \alpha_{\nu+\ell}(m)=0,\quad
	\nu=1,\dots,m-1;\qquad \alpha_m(m)=1.
\end{equation}
Note that \eqref{eq:2.17} has coefficients independent of $d$,
which also justifies that $\alpha_\ell(m)$ does not depends on $d$.

In order to remove the dependence of the upper bound of the sum in \eqref{eq:2.17} on $m-\nu$ we set $\alpha_\ell(m):=0$ for $\ell>m$.
Then \eqref{eq:2.17} becomes
\begin{equation}\label{eq:3.4}
\fD^\nu\alpha_\nu(m)=(-1)^\nu\delta_{\nu,m},\quad \nu=1,2,\dots,m,
\end{equation}
where $\delta_{\nu,m}$ is the Kroneker $\delta$ and
$\fD^\nu$ denotes the $\nu$th forward finite difference operator,
i.e. $\fD^\nu z_j:=\sum_{k=0}^\nu (-1)^{\nu+k}\binom{\nu}{k}z_{j+k}$.

We shall show that the solutions $\alpha_\nu(m)$ of \eqref{eq:3.4} for all $m\in\NN$
are uniquely determined by the following recursive procedure:
\begin{equation}\label{eq:3.1}
\alpha_k(m):=\delta_{k,m},\quad k\ge m,\quad m\in\NN;
\end{equation}
\begin{equation}\label{eq:3.2}
\alpha_k(m):=\alpha_{k+1}(m)+\alpha_{k-1}(m-1), \quad k=m-1,m-2,\dots,2,\quad m\ge 3;
\end{equation}
\begin{equation}\label{eq:3.3}
\alpha_1(m):=\alpha_2(m),\quad m\ge 2,
\end{equation}
where \eqref{eq:3.2} is applied inductively on $m$ and for given $m$ inductively on $k$.

In order to establish this we prove by induction on $m\in\NN$
that $\alpha_k(m)$, $k\in\NN$, from \eqref{eq:3.1}--\eqref{eq:3.3} satisfy \eqref{eq:3.4}.
Observe that \eqref{eq:3.4} trivially follows from \eqref{eq:3.3} for $\nu=1$, $m\ge 2$,
and from \eqref{eq:3.1} for $\nu=m$, $m\ge 1$.
Hence \eqref{eq:3.4} is true for $m=1$ and $m=2$.
For $m\ge 3$ assume \eqref{eq:3.4} is true for for $m-1$. Using  \eqref{eq:3.2} we get for $\nu=2,\dots,m-1$
\begin{multline*}
\fD^\nu\alpha_\nu(m)=\fD^{\nu-1}\alpha_{\nu+1}(m)-\fD^{\nu-1}\alpha_\nu(m)=-\fD^{\nu-1}(\alpha_\nu(m)-\alpha_{\nu+1}(m))\\
=-\fD^{\nu-1}\alpha_{\nu-1}(m-1)=0.
\end{multline*}
This verifies \eqref{eq:3.4} by induction.

Now, one establishes directly that the non-zero entries in \eqref{eq:3.1}--\eqref{eq:3.3} are given by \eqref{eq:3.5}
and hence \eqref{eq:3.5} solves \eqref{eq:2.17}.
This completes the proof of Lemma~\ref{lem:1}.
\end{proof}

\begin{rem}\label{rem:1}
The numbers $\alpha_\ell(m)$ from \eqref{eq:3.5} are known as \emph{ballot numbers}, see \cite[pp. 68, 76]{FS}.
The numbers $C_n=\alpha_1(n+1)$, $n=0,1,\dots$, are known as \emph{Catalan numbers}, see \cite[pp. 6, 17]{FS}.
Several values of $\alpha_\ell(m)$ are given in the following table.
\end{rem}

\begin{table}[ht]
	\centering
		\begin{tabular}{|l|rrrrrrrrrr|}
\hline
$m~\backslash~ \ell$&     1&     2&     3&     4&     5&     6&     7&     8&     9&     10\\
\hline
1&     1&     0&     0&     0&     0&     0&     0&     0&     0&     0\\
2&     1&     1&     0&     0&     0&     0&     0&     0&     0&     0\\
3&     2&     2&     1&     0&     0&     0&     0&     0&     0&     0\\
4&     5&     5&     3&     1&     0&     0&     0&     0&     0&     0\\
5&    14&    14&     9&     4&     1&     0&     0&     0&     0&     0\\
6&    42&    42&    28&    14&     5&     1&     0&     0&     0&     0\\
7&   132&   132&    90&    48&    20&     6&     1&     0&     0&     0\\
8&   429&   429&   297&   165&    75&    27&     7&     1&     0&     0\\
9&        1430&        1430&        1001&         572&         275&         110&          35&           8&           1&     0\\
10&        4862&        4862&        3432&        2002&        1001&         429&         154&          44&           9&           1\\
\hline
		\end{tabular}
\caption{$\alpha_\ell(m)$ for $1\le \ell, m \le 10$.}
\end{table}

\subsection{Completion of the proof of Theorem~\ref{thm:main}}

Using the fact
that $C_\ell^{(\mu)}$ and $T_\ell$ are even functions for even $\ell$ and odd functions for odd $\ell$
we rewrite the derivatives of the Newtonian kernel \eqref{Maxwell-3}--\eqref{Maxwell-4} as
\begin{align}
	\left(\eta\cdot\nabla\right)^\ell |x-a\eta|^{2-d}&= \ell!
	C_{\ell}^{(d/2-1)}\left(\frac{(a\eta-x)\cdot \eta}{|a\eta-x|}\right)|a\eta-x|^{-d+2-\ell},\label{eq:2.4}\\
	\left(\eta\cdot\nabla\right)^\ell \ln 1/|x-a\eta|&= (\ell-1)!
	T_{\ell}\left(\frac{(a\eta-x)\cdot \eta}{|a\eta-x|}\right)|a\eta-x|^{-\ell}.\label{eq:2.5}
\end{align}
By \cite[p. 442, (18.5.10)]{OLBC} for $\ell\ge 1$ we have
\begin{align}
C_{\ell}^{(d/2-1)}(t)
&=(d/2-1)\sum_{s=0}^{\left\lfloor \ell/2\right\rfloor}
\frac{(-1)^s(d/2)_{\ell-s-1}}{s!(\ell-2s)!}(2t)^{\ell-2s},\label{eq:2.6}\\
T_{\ell}(t)
&=\frac{\ell}{2}\sum_{s=0}^{\left\lfloor \ell/2\right\rfloor} \frac{(-1)^s(\ell-s-1)!}{s!(\ell-2s)!}(2t)^{\ell-2s}.\label{eq:2.7}
\end{align}
Now, by \eqref{eq:2.6} and \eqref{eq:2.4} substituted in the right-hand side of \eqref{eq:2.1} or
by \eqref{eq:2.7} and \eqref{eq:2.5} substituted in the right-hand side of \eqref{eq:2.2}
we get for $d\ge2$
\begin{multline}\label{eq:2.8}
	\FF_{\eps,m}(a\eta,x)=q_0|a\eta-x|^{2-d}\\
	+\sum_{\ell=1}^m
	q_{\ell}\delta^{\ell-1} a^\ell\sum_{s=0}^{\left\lfloor \ell/2\right\rfloor}
	\frac{(-1)^s(d/2)_{\ell-s-1}}{2s!(\ell-2s)!}
	\left(2\frac{(a\eta-x)\cdot \eta}{|a\eta-x|}\right)^{\ell-2s}
|a\eta-x|^{-d+2-\ell}.
\end{multline}
To find a convenient representation of the values of $\FF_{\eps,m}(a\eta,x)$ for $|x|=1$
we denote by $\theta$  the angle between the vectors $a\eta-x$ and $\eta$, $|a\eta-x|\cos\theta=(a\eta-x)\cdot\eta$.
By the Law of Cosines we have
\begin{equation*}
|a\eta-x|^2+a^2-2|a\eta-x|a\cos\theta=|x|^2,
\end{equation*}
which, with the notation
\begin{equation}\label{eq:2.9}
r:=|a\eta-x|/a,\quad  |x|=1,
\end{equation}
can be written as (recall $\delta=(a^2-1)/a^2$)
\begin{equation}\label{eq:2.10}
2\cos\theta=r+\delta r^{-1},\quad |x|=1.
\end{equation}

Using \eqref{eq:2.9},  \eqref{eq:2.10} and \eqref{eq:2.12} in \eqref{eq:2.8} we get
\begin{multline}\label{eq:2.11}
	\left.|a\eta-x|^{d-2}\FF_{\eps,m}(a\eta,x)\right|_{|x|=1}
	=q_0+\sum_{\ell=1}^m 	q_{\ell}\delta^{\ell-1} \sum_{s=0}^{\left\lfloor \ell/2\right\rfloor}
	b_{\ell,s}(r+\delta r^{-1})^{\ell-2s} r^{-\ell}\\
	=q_0+\sum_{\ell=1}^m 	q_{\ell}\delta^{\ell-1} \sum_{s=0}^{\left\lfloor \ell/2\right\rfloor}
	b_{\ell,s}(1+\delta r^{-2})^{\ell-2s} r^{-2s}=:A,
\end{multline}
where
\begin{equation}\label{eq:2.12}
b_{\ell,s}=b_{\ell,s}(d):=\frac{(-1)^s(d/2)_{\ell-s-1}}{2s!(\ell-2s)!}.
\end{equation}
We rewrite \eqref{eq:2.11} as follows.
In the expression
\begin{equation*}
A=q_0+\sum_{\ell=1}^m \sum_{s=0}^{\left\lfloor \ell/2\right\rfloor} \sum_{k=0}^{\ell-2s}
	q_{\ell} b_{\ell,s} \binom{\ell-2s}{k} \delta^{\ell+k-1}  r^{-2s-2k}
\end{equation*}
we set $k=\nu-s$ and get
\begin{multline*}
A
=q_0+\sum_{\ell=1}^m \sum_{s=0}^{\left\lfloor \ell/2\right\rfloor} \sum_{\nu=s}^{\ell-s}
	q_{\ell} b_{\ell,s} \binom{\ell-2s}{\nu-s} \delta^{\ell+\nu-s-1}  r^{-2\nu}\\
=q_0+\sum_{\ell=1}^m \sum_{\nu=0}^{\ell} \sum_{s=0}^{\min\{\nu,\ell-\nu\}}
	q_{\ell} b_{\ell,s} \binom{\ell-2s}{\nu-s} \delta^{\ell+\nu-s-1}  r^{-2\nu}.
\end{multline*}
Separating the terms for $\nu=0$ (which implies $s=0$) and
shifting the order of summation in $\ell$ and $\nu$ in the triple sum above we get
\begin{equation*}
A=q_0+\sum_{\ell=1}^m q_{\ell} b_{\ell,0} \delta^{\ell-1}
+\sum_{\nu=1}^{m} \sum_{\ell=\nu}^m \sum_{s=0}^{\min\{\nu,\ell-\nu\}} \!\!
	q_{\ell} b_{\ell,s} \binom{\ell-2s}{\nu-s} \delta^{\ell+\nu-s-1}  r^{-2\nu}.
\end{equation*}
In the triple sum we set $s=\ell-\nu-k$ with $(\ell-2\nu)_+\le k \le \ell-\nu$ and get
\begin{multline}\label{eq:2.13}
A=q_0+\sum_{\ell=1}^m q_{\ell} b_{\ell,0} \delta^{\ell-1}\\
+\sum_{\nu=1}^{m} \sum_{\ell=\nu}^m q_{\ell} \left[\sum_{k=(\ell-2\nu)_+}^{\ell-\nu}
	b_{\ell,\ell-\nu-k} \binom{2\nu-\ell+2k}{k} \delta^k\right]\delta^{2\nu-1}  r^{-2\nu}\\
=q_m b_{m,0}\delta^{2m-1}r^{-2m},
\end{multline}
where we used \eqref{eq:2.14b} for the last equality.
Indeed, if the summation index in the $\nu+1$-st row of \eqref{eq:2.14b} is changed from
$\ell$ to $\ell-\nu$ and  this equation is multiplied by $b_{\nu,0}$ then
\begin{equation*}
	\sum_{\ell=\nu}^m q_{\ell} \left[\sum_{k=(\ell-2\nu)_+}^{\ell-\nu}
	b_{\ell,\ell-\nu-k} \binom{2\nu-\ell+2k}{k} \delta^k\right]=0.
\end{equation*}
Now, \eqref{eq:2.13}, $q_m=1$, \eqref{eq:2.11} and \eqref{eq:2.9} yield
$\cF_{\eps}(x\cdot\eta)=\sF_{\eps,m}(a\eta,x)$ for $x\in\SS$,
which completes the proof of Theorem~\ref{thm:main}.
\qed

\smallskip

The asymptotic of $\sF_{\eps,m}(a\eta,0)$ as $\eps\to 0$
(and  of $\int_{\SS} \cF_{\eps}(x\cdot\eta)\,d\sigma(x)$ as well) is given by


\begin{prop}
Under the assumptions of \thmref{thm:main} we have
\begin{multline}\label{eq:2.40}
\sF_{\eps,m}(a\eta,0)
=a^{2-d}\left(q_0+\sum_{\ell=1}^m
	\frac{(\ell+d-3)!}{\ell!(d-2)!}q_{\ell}\delta^{\ell-1}\right)\\
=\frac{1}{2}\alpha_{1,0}(m)+ O(\eps)=\frac{1}{2}\frac{(2m-2)!}{m!(m-1)!}+ O(\eps).
\end{multline}
\end{prop}

\begin{proof}
In order to evaluate $\sF_{\eps,m}(a\eta,0)$ in the case $d\ge 3$ we substitute \eqref{eq:2.4} in \eqref{eq:2.1} and
use that $C_{\ell}^{(d/2-1)}(1)=\binom{\ell+d-3}{\ell}$ to get
\begin{equation}\label{eq:2.15}
	\sF_{\eps,m}(a\eta,0)=a^{2-d}\left(q_0+\sum_{\ell=1}^m
	\frac{(\ell+d-3)!}{\ell!(d-2)!}q_{\ell}\delta^{\ell-1}\right)=q_0+q_1+O(\eps).
\end{equation}
The validity of \eqref{eq:2.15} in the case $d=2$ is obtain
by substituting \eqref{eq:2.5} in \eqref{eq:2.2} and the use of $T_\ell(1)=1$.
Note that \eqref{eq:2.15} is the first equality in \eqref{eq:2.40}.

From the first equation of \eqref{eq:2.14b} we get $q_0=-q_1/2+O(\eps)$,
which together with \eqref{eq:2.15} and \eqref{eq:2.0} gives
\begin{equation}\label{eq:2.18}
\sF_{\eps,m}(a\eta,0)=q_1/2+O(\eps)=\alpha_{1,0}/2+O(\eps).
\end{equation}
Finally, \eqref{eq:2.18} and \eqref{eq:3.5} with $\ell=1$ prove \eqref{eq:2.40}.
\end{proof}

\begin{rem}\label{rem:3}
The values of the $q_\ell$'s and $\sF_{\eps,m}(a\eta,0)$ for $m=1,2,3,4$ are as follows:
\begin{itemize}
	\item If $m=1$, then $q_0=-\frac{1}{2}$, $q_1=1$,
\begin{equation*}
a^{d-2}\sF_{\eps,1}(a\eta,0)=\frac{1}{2}.	
\end{equation*}
Note that $\sF_{\eps,1}(a\eta,x)=\frac{1}{2}(a-|x|^2)/|a\eta-x|^{d}$,
i.e. $\lim_{\eps\to 0}\sF_{\eps,1}(a\eta,x)$ is a~constant multiple of the Poisson kernel.
	\item If $m=2$, then $q_0=-\frac{1}{2}+\frac{d}{8}\delta$, $q_1=1-\frac{d}{2}\delta$, $q_2=1$,
\begin{equation*}
a^{d-2}\sF_{\eps,2}(a\eta,0)=\frac{1}{2}+\frac{d-4}{8}\delta.	
\end{equation*}
	\item If $m=3$, then $q_0=-1+\frac{d+2}{4}\delta-\frac{d(d+2)}{48}\delta^2$, $q_1=2-(d+1)\delta+\frac{d(d+2)}{8}\delta^2$,
	$q_2=2-\frac{d+2}{2}\delta$, $q_3=1$,
\begin{equation*}
a^{d-2}\sF_{\eps,3}(a\eta,0)=1+\frac{d-6}{4}\delta+\frac{(d-4)(d-6)}{48}\delta^2.	
\end{equation*}
	\item If $m=4$, then $q_0=-\frac{5}{2}+\frac{5(d+4)}{8}\delta-\frac{(d+2)(d+4)}{16}\delta^2+\frac{d(d+2)(d+4)}{384}\delta^3$, \\
	$q_1=5-\frac{5(d+2)}{2}\delta+\frac{(3d+2)(d+4)}{8}\delta^2-\frac{d(d+2)(d+4)}{48}\delta^3$, \\
	$q_2=5-\frac{3d+10}{2}\delta+\frac{(d+2)(d+4)}{8}\delta^2$, $q_3=3-\frac{d+4}{2}\delta$, $q_4=1$,
\begin{equation*}
a^{d-2}\sF_{\eps,4}(a\eta,0)=\frac{5}{2}+\frac{5(d\!-\!8)}{8}\delta+\frac{(d\!-\!6)(d\!-\!8)}{16}\delta^2
	+\frac{(d\!-\!4)(d\!-\!6)(d\!-\!8)}{384}\delta^3.	
\end{equation*}
\end{itemize}
\end{rem}

\section{Localized kernels on $\SSS^1$: Second solution}\label{sec:S2}

In dimension $d=2$ we next identify
another linear combination of a single shift of  the Newtonian kernel directional derivatives
with excellent localization on the unit sphere $\SSS^1$.

\begin{thm}\label{thm:main-3}
Let $0<\eps\le 1$, $a=e^\eps$, $m\in\NN$, and $\eta\in {\mathbb S}^1$.
The function
\begin{equation}\label{eq:5.5}
\GG_{\eps}(x\cdot\eta):=\frac{2^{2m-2}}{m}\sum_{n\in\ZZ} \frac{\eps^{-1}}{(1+\eps^{-2}(\rho(x,\eta)+2\pi n)^2)^{m}},
\quad x\in {\mathbb S}^1,
\end{equation}
has the following properties:
\begin{equation}\label{eq:5.7}
0<\GG_{\eps}(x\cdot\eta)\le c \frac{\eps^{-1}}{(1+\eps^{-1}\rho(x,\eta))^{2m}}, \quad x\in\SSS^1,
\end{equation}
with a constant $c>0$ depending only on $m$,
and
\begin{equation}\label{eq:5.6}
\int_{\SSS^1} \GG_{\eps}(x\cdot\eta)\,d\sigma(x)=\frac{\pi(2m-2)!}{(m-1)!m!}.
\end{equation}
Moreover,
$\GG_{\eps}(x\cdot\eta)$ is the restriction to $\SSS^1$ of
the following harmonic function, defined on $\RR^2\backslash\{a\eta\}$,
\begin{equation}\label{eq:5.1}
	\sG_{\eps,m}(a\eta,x):=-\frac{1}{2}\frac{(2m-2)!}{m!(m-1)!}
	+\sum_{\ell=1}^m Q_\ell(2\eps)\frac{(2\eps)^{\ell-1}a^{\ell}}{\ell!}
	(\eta\cdot\nabla)^\ell \ln \frac{1}{|x-a\eta|},
\end{equation}
where
\begin{equation}\label{eq:5.2}
Q_\ell(u)
=\sum_{k=\ell}^{m}\frac{\ell(2m-k-1)!}{m!(m-k)!}\frac{A_{k-1,\ell-1}}{(k-1)!}u^{k-\ell},\quad 1\le \ell\le m,
\end{equation}
\begin{equation}\label{eq:5.3}
A_{k,\ell}=\sum_{\nu=\ell}^k (-1)^{\nu-\ell}\binom{\nu}{\ell}\nu!S_{k,\nu},\quad 0\le \ell\le k,
\end{equation}
and $S_{k,\nu}$ denote the Stirling numbers of the second kind, defined by
\begin{equation}\label{eq:5.4}
u^k=\sum_{\nu=0}^k S_{k,\nu}u(u-1)\cdots(u-\nu+1),\quad k=0,1,\dots.
\end{equation}
Note that $S_{k,0}=\delta_{k,0}$, $S_{k,k}=1$.

\end{thm}

The proof of \thmref{thm:main-3} is based on several auxiliary statements.

\begin{lem}\label{lem:2}
Let $m\in\NN$, $\eps>0$ and $\GGG_m(u):=2\pi \eps^{-1}(1+(2\pi \eps^{-1}u)^2)^{-m}$ for $u\in\RR$.
Then the Fourier transform of $\GGG_m$ has the representation
\begin{equation}\label{eq:5}
\hat{\GGG}_m(v) :=\int_{\RR} \GGG_m(u)e^{-iuv}\,du
=  e^{-|v|\eps/(2\pi)} \sum_{k=1}^{m} \beta_{m-1,k-1} \left(\frac{|v|\eps}{2\pi}\right)^{k-1},
\end{equation}
where
\begin{equation}\label{eq:beta}
\beta_{m,k}:=\frac{\pi(2m-k)!2^k}{k!(m-k)!m!2^{2m}},\quad 0\le k\le m.
\end{equation}
\end{lem}

\begin{proof}
We have $\GGG_m(u)=b^{2m-1}(b^2+u^2)^{-m}$ with $b:=\eps/(2\pi)$.
Clearly, the function $\GGG_m$ is even. Hence, it suffices to prove \eqref{eq:5} only for $v\ge 0$.
From identity 1.3.7 in \cite[p.11]{E} (which gives the Fourier cosine transform) with $\nu=m-\frac{1}{2}$ we get
\begin{equation}\label{eq:2}
\hat{\GGG}_m(v) = 2b^{2m-1}\pi^{1/2}\left(\frac{v}{2b}\right)^{m-1/2} \Gamma(m)^{-1} K_{m-1/2}(bv),
\end{equation}
where $K_{m-1/2}$ is the modified Bessel function of the second kind for half an odd integer index.
According to identity 10.47.9 in \cite[p. 262]{OLBC} $K_{m-1/2}$ is related to
the modified spherical Bessel function $\mathbf{\textsf{k}}_{m-1}$ by
\begin{equation}\label{eq:3}
 K_{m-1/2}(z) = \sqrt{\frac{2z}{\pi}}\mathbf{\textsf{k}}_{m-1}(z)
\end{equation}
and $\mathbf{\textsf{k}}_{m-1}$ has the explicit form (identities 10.49.12 and 10.49.1 in \cite[p.264]{OLBC})
\begin{equation}\label{eq:4}
 \mathbf{\textsf{k}}_{m-1}(z) = \frac{\pi}{2}e^{-z}\sum_{\nu=0}^{m-1} \frac{(m-1+\nu)!}{(m-1-\nu)!\nu!2^\nu}z^{-\nu-1}.
\end{equation}
Now \eqref{eq:5} follows from \eqref{eq:2}, \eqref{eq:3} and \eqref{eq:4}.
\end{proof}

\begin{lem}\label{lem:4}
Let $k\in\NN$ and $t\in\CC$, $|t|<1$. Then
\begin{equation}\label{eq:8}
\sum_{n=0}^\infty n^{k-1} t^n
=\sum_{\ell=1}^k A_{k-1,\ell-1} (1-t)^{-\ell},
\end{equation}
where $A_{k,\ell}$ are defined in \eqref{eq:5.3}.
\end{lem}

\begin{proof}
Identity \eqref{eq:8} for $k=1$ reduces to the geometric series $\sum_{n=0}^\infty t^n=(1-t)^{-1}$. Let $k\ge 2$.
We differentiate the previous identity $\nu$ times, then multiply by $t^\nu$
and finaly apply the binomial formula to obtain
\begin{multline*}
\sum_{n=0}^\infty n(n-1)\dots(n-\nu+1)t^n=\nu! t^\nu (1-t)^{-\nu-1}\\
=\nu! \sum_{\ell=0}^\nu \binom{\nu}{\ell} (t-1)^\ell (1-t)^{-\nu-1}
=\nu! \sum_{\ell=0}^\nu (-1)^{\nu-\ell}\binom{\nu}{\ell} (1-t)^{-\ell-1}.
\end{multline*}
This coupled with \eqref{eq:5.4}, where $k$ replaced by $k-1$, leads to
\begin{multline*}
\sum_{n=0}^\infty n^{k-1} t^n=\sum_{\nu=0}^{k-1} S_{k-1,\nu}\sum_{n=0}^\infty n(n-1)\dots(n-\nu+1)t^n\\
=\sum_{\nu=0}^{k-1} S_{k-1,\nu}\nu! \sum_{\ell=0}^\nu (-1)^{\nu-\ell}\binom{\nu}{\ell} (1-t)^{-\ell-1},
\end{multline*}
which proves the lemma.
\end{proof}

\begin{thm}\label{thm:1}
Let $m\in\NN$, $\eps>0$, $a=e^\eps$, $\GGG_m(u):=2\pi \eps^{-1}(1+(2\pi \eps^{-1}u)^2)^{-m}$ for $u\in\RR$, and $z=e^{-2\pi i u}$.
Then
\begin{multline}\label{eq:9}
\sum_{\nu\in\ZZ} \GGG_m(\nu+u)\\
= -\beta_{m-1,0} +2\sum_{\ell=1}^{m} \sum_{k=\ell}^{m} \beta_{m-1,k-1}A_{k-1,\ell-1} \eps^{k-1}a^\ell
Re\left\{(a-z)^{-\ell}\right\}.
\end{multline}
\end{thm}

\begin{proof}
Applying \lemref{lem:2} and the Poisson summation formula:
\begin{equation*}
	\sum_{\nu=-\infty}^\infty \GGG_m(\nu+u)=	\sum_{n=-\infty}^\infty \hat{\GGG}_m(2\pi n) e^{-2\pi i n u}
\end{equation*}
we get
\begin{equation}\label{eq:6}
\sum_{\nu=-\infty}^\infty \GGG_m(\nu+u)
=  \sum_{k=1}^{m} \beta_{m-1,k-1}\eps^{k-1} \sum_{n=-\infty}^\infty |n|^{k-1} e^{-|n|\eps} e^{-2\pi i n u}.
\end{equation}
For the evaluation of the inner sum in the right-hand side of \eqref{eq:6}
we use Lemma~\ref{lem:4} with $t=a^{-1}z$ and with $t=a^{-1}\bar{z}$ to get
\begin{multline}\label{eq:8b}
\sum_{n=-\infty}^\infty |n|^{k-1} e^{-|n|\eps} e^{-2\pi i n u}\\
=\sum_{\ell=1}^k A_{k-1,\ell-1} \left[(1-a^{-1}z)^{-\ell}+(1-a^{-1}\bar{z})^{-\ell}-\delta_{k,1}\right]\\
=-\delta_{k,1}+\sum_{\ell=1}^k A_{k-1,\ell-1} 2a^\ell Re\left\{(a-z)^{-\ell}\right\}.
\end{multline}
Substituting \eqref{eq:8b} in \eqref{eq:6} we arrive at \eqref{eq:9}.
\end{proof}

\begin{proof}
[Proof of \thmref{thm:main-3}]
Due to the rotational invariance we may assume that the vector $\eta=(1,0)$ in \eqref{eq:5.5}.
For any $x=(x_1,x_2)\in\SSS^1$ we apply \thmref{thm:1} with $z=x_1+ix_2=e^{-2\pi i u}$, $|u|\le 1/2$.
Thus $\rho(x,\eta)=2\pi |u|$ and $a-z=|x-a\eta|e^{i\varphi}$, where $\cos\varphi=-\frac{(x-a\eta)\cdot \eta}{|x-a\eta|}$.
Using the Maxwell formula \eqref{Maxwell-4} we get
\begin{multline}\label{eq:5.8}
Re\left\{(a-z)^{-\ell}\right\}
=Re\left\{(a-\bar{z})^{-\ell}\right\}\\
= (-1)^{\ell} T_{\ell}\left(\frac{(x-a\eta)\cdot \eta}{|x-a\eta|}\right)\frac{1}{|x-a\eta|^\ell}
=\frac{1}{(\ell-1)!}\left(\eta\cdot\nabla\right)^\ell \ln \frac{1}{|x-a\eta|}.
\end{multline}
Now, combining \eqref{eq:5.8} with \eqref{eq:9} we get
\begin{multline}\label{eq:5.9}
\sum_{\nu=-\infty}^\infty 2\pi \eps^{-1}(1+\eps^{-2}(\rho(x,\eta)+2\pi \nu)^2)^{-m}\\
= -\beta_{m-1,0} +2\sum_{\ell=1}^{m} \sum_{k=\ell}^{m} \beta_{m-1,k-1}A_{k-1,\ell-1} \eps^{k-1}a^\ell
\frac{1}{(\ell-1)!}\left(\eta\cdot\nabla\right)^\ell \ln \frac{1}{|x-a\eta|}
\end{multline}
whenever $u\ge 0$.
Identity \eqref{eq:5.9} is also valid for $u< 0$
because the left-hand side of \eqref{eq:9} is an even function of $u$.
Now, multiplying both sides of \eqref{eq:5.9} by $\frac{2^{2m-3}}{\pi m}$ we obtain \eqref{eq:5.5}.

Inequalities \eqref{eq:5.7} follow readily by \eqref{eq:5.5}.
From \eqref{eq:5.5} and \eqref{eq:5} we get
\begin{equation*}
\int_{\SSS^1}\sG_{\eps,m}(a\eta,x)\,d\sigma(x)
=\frac{2^{2m-2}}{m}\int_{\RR} (1+u^2)^{-m}du=\frac{2^{2m-2}}{m}\hat{\GGG}_m(0)=\frac{\pi(2m-2)!}{(m-1)!m!},
\end{equation*}
which confirms \eqref{eq:5.6}.
\end{proof}

\begin{rem}
Some similarities and differences between the functions
$\sF_{\eps,m}(a\eta,x)$ defined in \eqref{eq:2.2} and $\sG_{\eps,m}(a\eta,x)$ defined by \eqref{eq:5.1} are:
\begin{itemize}
	\item $\sF_{\eps,m}$ is defined for every $d\ge 2$, while $\sG_{\eps,m}$ is defined only for $d=2$.

	\item In \eqref{eq:2.2} $a=1+\eps$, while $a=e^\eps$ in \eqref{eq:5.1}.

	\item $q_\ell$ and $Q_\ell$ are polynomials of the same degree and $q_\ell(\delta)-Q_\ell(2\eps)= O(\eps)$, $\ell=1,\dots,m$.

	\item The polynomials $Q_\ell$ are given explicitly, while the $q_\ell$'s are only known recursively.
\end{itemize}
\end{rem}

\section{Localization on $\SS$ of harmonic functions on $\RR^d\setminus \overline{B^d}$}\label{sec:S4}

Having solved Problem 1 one can easily solved the analogous problem for localization on $\SS$
of linear combinations of shifts of the Newtonian kernel with poles inside the unit ball.
The answer is given by the simple

\begin{prop}
For $d>2$, $\eta\in\SS$ and $a_\nu>1$ the harmonic functions on $\RR^d\setminus\cup_{\nu=1}^m\{a_\nu\eta\}$
\begin{equation*}
b_0 + \sum_{\nu=1}^{m} \frac{b_\nu}{|x-a_\nu\eta|^{d-2}}
\end{equation*}
and the harmonic functions on $\RR^d\setminus\cup_{\nu=1}^m\{a_\nu^{-1}\eta\}$
\begin{equation*}
b_0 + \sum_{\nu=1}^{m} \frac{b_\nu|a_\nu|^{2-d}}{|x-a_\nu^{-1}\eta|^{d-2}}
\end{equation*}
coincide on $\SS$.

For $d=2$, $\eta\in\SSS^1$ and $a_\nu>1$ the harmonic functions on $\RR^2\setminus\cup_{\nu=1}^m\{a_\nu\eta\}$
\begin{equation*}
b_0 + \sum_{\nu=1}^{m} b_\nu\ln \frac{1}{|x-a_\nu\eta|}
\end{equation*}
and the harmonic functions on $\RR^2\setminus\cup_{\nu=1}^m\{a_\nu^{-1}\eta\}$
\begin{equation*}
b_0 + \sum_{\nu=1}^{m}b_\nu \ln\frac{1}{|a_\nu|} + \sum_{\nu=1}^{m} b_\nu\ln \frac{1}{|x-a_\nu^{-1}\eta|}
\end{equation*}
coincide on $\SSS^1$.
\end{prop}
The proof follows immediately by the symmetry lemma:
\begin{equation*}
a|x-a^{-1}\eta|=|x-a\eta|,\quad x,\eta\in\SS,~a>0.
\end{equation*}

\section{Localized kernels on $\RR^{d-1}$ in terms of Newtonian kernels}\label{sec:Rd}

In this section we construct highly localised kernels on the subspace
\begin{equation}\label{def-Rd-1}
\RR^{d-1}:=\{x\in\RR^d: x=(x_1, \dots, x_{d-1},0)\}
\quad\hbox{of $\;\RR^d$.}
\end{equation}
In this case the problem is less involved compared to the case on $\SS$ and the solution is simpler.


\begin{thm}\label{thm:main-2}
Let $m\in\NN$, $d\ge 2$, $\eps>0$, and $\eta=(0,\dots,0,-1)$.
Denote
\begin{equation}\label{def-F-eps}
\cF_{\eps,m}^*(x) := \frac{2^{2m-2}(d/2)_{m-1}}{m!}\frac{\eps^{2m-1}}{|x-\eps\eta|^{2m+d-2}},
\quad x\in\RR^{d-1}.
\end{equation}
The function $\cF_{\eps,m}^*$ has the following properties:
\begin{equation}\label{main-22}
0< \cF_{\eps,m}^*(x) \le \frac{c_1\eps^{-d+1}}{(1+\eps^{-1}|x|)^{2m+d-2}},
\quad\forall x\in \RR^{d-1},
\end{equation}
and
\begin{equation}\label{main-23}
\int_{\RR^{d-1}}\cF_{\eps,m}^*(x)\,dx_1\dots dx_{d-1} \ge c_2>0,
\end{equation}
where $c_1, c_2>0$ are constants depending only on $m$ and $d$.
Furthermore, $\cF_{\eps,m}^*$ is the restriction to $\RR^{d-1}$
of the harmonic function $\sF_{\eps,m}^*$, defined on $\RR^{d}\setminus \{\eps\eta\}$,
\begin{equation}\label{main-24}
\sF_{\eps,m}^*(x)
=\sum_{\ell=1}^m \frac{(-1)^\ell 2^{\ell-1}(2m-\ell-1)!}{(\ell-1)!m!(m-\ell)!(d-2)}\eps^{\ell-1} \partial_d^\ell \frac{1}{|x-\eps\eta|^{d-2}}
\quad\mbox{if}~d\ge3,
\end{equation}
\begin{equation}\label{main-25}
\sF_{\eps,m}^*(x)
= \sum_{\ell=1}^m \frac{(-1)^{\ell} 2^{\ell-1}(2m-\ell-1)!}{(\ell-1)!m!(m-\ell)!}\eps^{\ell-1} \partial_d^\ell \ln \frac{1}{|x-\eps\eta|}
\quad\mbox{if}~d=2,
\end{equation}
where $\partial_d$ stands for the partial derivative with respect to $x_d$.
\end{thm}

From \thmref{thm:main-2} we immediately get

\begin{cor}\label{cor:main-2}
Under the hypotheses of Theorem~\ref{thm:main-2} define
\begin{equation}\label{def-F-2}
F_{\eps,m}^*(x):=\cF_{\eps,m}^*(x)\Big(\int_{\RR^{d-1}}\cF_{\eps,m}^*(y)dy\Big)^{-1}, \quad x\in\RR^{d-1}.
\end{equation}
Then $F_{\eps,m}^*(x)$ is a summability kernel with decay just as in $(\ref{main-22})$
that can be represented as a linear combination of
$\partial_d^\ell |x-\eps\eta|^{2-d}$ if $d >2$ or
$\partial_d^\ell \ln 1/|x-\eps\eta|$ if $d >2$ for $\ell=1,\dots, m$.
\end{cor}

\begin{proof}[Proof of \thmref{thm:main-2}]
We shall derive this result from Theorem~\ref{thm:main} by a limiting process.

Our first step is to obtain a version of \thmref{thm:main} for an arbitrary sphere of radius $R$ in $\RR^d$.
Let $m\in\NN$, $d\ge 2$, $\eps>0$, $\eta\in\SS$, $\bar{x}\in\RR^d$, and $R>\eps$.
Set $\bar{y}=\bar{x}+(R+\eps)\eta$.
Denote by $\SSS(\bar{x}, R)$ the sphere in $\RR^d$ centered at $\bar{x}$ of radius $R$,
i.e. $\SSS(\bar{x}, R):= \{\bar{x}\}+R\SS$.
Scaling by a factor of $1/R$ the sphere $\SSS(\bar{x}, R)$ and the pole location $\bar{y}$
we arrive at the sphere $\SSS(\bar{x}/R, 1)$
and pole location at $\bar{y}/R=\bar{x}/R+(1+\eps/R)\eta$.
By (\ref{eq:2.3}) with $\eps/R$ and $\bar{x}/R$ in the place of $\eps$ and $\bar{x}$
we get for $x/R\in \SSS(\bar{x}/R, 1)$
\begin{equation*}
\cF_{\eps/R,m}\left(\frac{x-\bar{x}}{R}\cdot\eta\right)
= \frac{(d/2)_{m-1}}{2m!}\frac{(\eps/R)^{2m-1}(2+\eps/R)^{2m-1}}{(1+\eps/R)^{2m-2}}
\left|\frac{x}{R}-\frac{\bar{y}}{R}\right|^{-d+2-2m}.
\end{equation*}
We multiply both sides above by $R^{1-d}$ and factor $1/R$ out of the norm to obtain
\begin{equation*}
R^{1-d}\cF_{\eps/R,m}\left(\frac{x-\bar{x}}{R}\cdot\eta\right)
= \frac{(d/2)_{m-1}}{2m!}\frac{(2+\eps/R)^{2m-1}}{(1+\eps/R)^{2m-2}}\eps^{2m-1}|x-\bar{y}|^{-d+2-2m}.
\end{equation*}
Now, using \thmref{thm:main} and \eqref{eq:2.0} we obtain the follow representations
of the functions $R^{1-d}\cF_{\eps/R,m}((x-\bar{x})R^{-1}\cdot\eta)$ for $x\in \SSS(\bar{x},R)$:
In the case $d\ge 3$ we have
\begin{align}\label{eq:2.3c}
&\frac{(d/2)_{m-1}}{2m!}\frac{(2+\eps/R)^{2m-1}}{(1+\eps/R)^{2m-2}}\eps^{2m-1}|x-\bar{y}|^{-d+2-2m} \notag
\\
&=\sum_{k=0}^{m-1}\alpha_{0,k}\frac{(\eps/R)^{k}(2+\eps/R)^{k}}{(1+\eps/R)^{2k}}R^{-1}|x-\bar{y}|^{2-d}
\\
& +\sum_{\ell=1}^m \frac{(2+\eps/R)^{\ell-1}}{(1+\eps/R)^{\ell-2}}\frac{\eps^{\ell-1}}{\ell!(d-2)}
\sum_{k=0}^{m-\ell}\alpha_{\ell,k}\frac{(\eps/R)^{k}(2+\eps/R)^{k}}{(1+\eps/R)^{2k}} (\eta\cdot\nabla)^\ell |x-\bar{y}|^{2-d} \notag
\end{align}
and in the case $d=2$
\begin{multline}\label{eq:2.3d}
\frac{1}{2m}\frac{(2+\eps/R)^{2m-1}}{(1+\eps/R)^{2m-2}}\frac{\eps^{2m-1}}{|x-\bar{y}|^{2m}}
=\sum_{k=0}^{m-1}\alpha_{0,k}\frac{(\eps/R)^{k}(2+\eps/R)^{k}}{(1+\eps/R)^{2k}}R^{-1}\\
+\sum_{\ell=1}^m \frac{(2+\eps/R)^{\ell-1}}{(1+\eps/R)^{\ell-2}}\frac{\eps^{\ell-1}}{\ell!}
\sum_{k=0}^{m-\ell}\alpha_{\ell,k}\frac{(\eps/R)^{k}(2+\eps/R)^{k}}{(1+\eps/R)^{2k}} (\eta\cdot\nabla)^\ell \ln\frac{1}{|x-\bar{y}|}.
\end{multline}

We are prepared to prove identities (\ref{main-24})--(\ref{main-25}).
Let $\eps>0$ and $\eta=(0,\dots,0,-1)$.
Fix $x^\star=(x_1^\star,\dots, x_{d-1}^\star, 0)\in\RR^{d-1}$ (see (\ref{def-Rd-1}))
and let $R>\max\{|x^\star|,\eps\}$.

We choose $\bar{x}:=-R\eta$, $\bar{y}:=\eps\eta=(0,\dots,0,-\eps)$,
and consider the point $x\in \SSS(\bar{x},R)$ defined by
$$
x:= (x_1^\star,\dots, x_{d-1}^\star, x_d), \quad\hbox{where}\quad x_d:=|x^\star|^2/(R+\sqrt{R^2-|x^\star|^2}).
$$
It is easy to verify that $x-\bar{x}\in R\SS$. Then \eqref{eq:2.3c} and \eqref{eq:2.3d} hold.
Letting $R\to\infty$ in \eqref{eq:2.3c} or \eqref{eq:2.3d}, using \lemref{lem:1} and observing that $x\to x^*$
we conclude that the restriction of $\sF_{\eps,m}^*$ from \eqref{main-24}--\eqref{main-25}
coincides with $\cF_{\eps,m}^*$ from \eqref{def-F-eps} at every point $x^\star\in \RR^{d-1}$.

Inequalities \eqref{main-22}--\eqref{main-23} follow trivially from \eqref{def-F-eps}.
\end{proof}

\end{document}